\documentclass{article}

\usepackage{amsmath,amssymb}

\usepackage{diagbox}

\usepackage{mathrsfs}
\usepackage[amsmath, thmmarks]{ntheorem}

\usepackage{hyperref}

\hypersetup{colorlinks=true,allcolors=[rgb]{0,0.4,1},citecolor=red}

\usepackage{tikz}
\usepackage{graphicx}
\usetikzlibrary{calc,decorations.markings}

\usepackage{array}
\usepackage{makecell}

\usepackage{multirow}

\newtheorem{theorem}{Theorem}[section]
\newtheorem{lemma}[theorem]{Lemma}

\newtheorem{claim}[theorem]{Claim}

\theorembodyfont{\normalfont}
\newtheorem{definition}[theorem]{Definition}
\newtheorem*{remark}{Remark}
\newtheorem{example}[theorem]{Example}

\newtheorem*{result}{Result}

\theoremheaderfont{\itshape}
\theoremseparator{.}
\theoremsymbol{\ensuremath{\blacksquare}}
\newtheorem*{proof}{Proof}

\newcommand{\BC}{\mathbb{C}}

\newcommand{\BZ}{\mathbb{Z}}

\newcommand{\fg}{\mathfrak{g}}

\newcommand{\fh}{\mathfrak{h}}

\newcommand{\fn}{\mathfrak{n}}

\DeclareMathOperator{\End}{End}
\newcommand{\SL}{\mathfrak{sl}}
\newcommand{\GL}{\mathfrak{gl}}
\DeclareMathOperator{\str}{sTr}
\DeclareMathOperator{\Tr}{Tr}

\title{On the Lie superalgebra $\GL(m|n)$ weight system}
\author{Zhuoke Yang\thanks{Higher School of Economics, partially supported by International Laboratory of Cluster Geometry NRU HSE, RF Government grant, ag. № 075-15-2021-608 dated 08.06.2021}
\\ \textsf{zetadkyzk@gmail.com}}
\date{}

\begin{document}
\maketitle
\begin{abstract}
To a finite type knot invariant, a weight system can be associated,
which is a function on chord diagrams satisfying so-called $4$-term relations.
In the opposite direction, each weight system determines a finite type
knot invariant. In particular, a weight system can be associated to any
metrized Lie algebra,
and any metrized Lie superalgebra. However, computation of these weight systems
is complicated.
In the recent paper by the present author, an extension of
the $\GL(N)$-weight system to arbitrary permutations is defined, which
allows one to develop a recurrence relation for an efficient computation of its values.
In addition, the result proves to be universal, valid for all values of~$N$
and allowing thus to define a unifying $\GL$-weight system taking values in the ring of polynomials
in infinitely many variables $C_0=N,C_1,C_2,\dots$.
In the present paper,
we extend this construction to the weight system associated to the Lie superalgebra $\GL(m|n)$.
Then we prove that the $\GL(m|n)$-weight system is equivalent to the $\GL$-one,
under the substitution $C_0=m-n$.
\end{abstract}

\section{Introduction}
In V.~A.~Vassiliev's theory of finite type knot invariants~\cite{vassiliev1990cohomology},
a weight system can be associated to each such invariant.
A weight system is a function on chord diagrams satisfying so-called
$4$-term relations.

In the opposite direction, according to a Kontsevich theorem~\cite{kon1993},
to each weight system taking values in a field of characteristic~$0$,
a finite type knot invariant can be associated in a canonical way.
This makes studying weight systems an important part of knot theory.

There is a number of approaches to constructing weight systems. In particular,
a huge class of weight systems can be constructed from metrized finite dimensional Lie algebras.
In spite of the fact that the construction is straightforward,
explicit computations are elaborative, and until recently no efficient
way to implement them was known.
In a recent paper~\cite{ZY2}, the present author, following a suggestion of M.~Kazarian,
extended the weight system corresponding to the Lie algebra~$\GL(N)$
to arbitrary permutations, which allowed for proving a recurrence relation
for it, whence computing its values explicitly.
By means of the recurrence relation, we have defined a universal $\GL$-weight system,
which contains in itself all the $\GL(N)$-weight systems, for arbitrary~$N$.

In the present paper, we do a similar thing  for
the weight system corresponding to the Lie superalgebra $\GL(m|n)$.
We prove that it is a specialization of the $\GL$-weight system, for $C_0=m-n$.

The original references to the Lie superalgebras can be found in~\cite{Kac}.
Weight systems arising from Lie superalgebras are defined in~\cite{Va}.
The straightforward approach to computing the values of a Lie superalgebra weight system
on a general chord diagram amounts to elaborating calculations in the noncommutative
universal enveloping algebra, in spite of the fact that the result belongs to the
center of the latter. This approach is rather inefficient even for the simplest
noncommutative Lie superalgebra $\GL(1|1)$. For this Lie Superalgebra, however, there is
a recurrence relation due to Figueroa-O’Farrill, T.~Kimura and A.~Vaintrob~\cite{FKV}.
Much less is known about other Lie superalgebras; the goal of the present paper is to
establish an efficient way to compute the $\GL(m|n)$-weight system, for arbitrary~$m$ and~$n$,
and to prove that it is equivalent to the $\GL$-weight system.

The approach is based on defining an invariant of permutations taking values in the center
of the universal enveloping algebra of $\GL(m|n)$. The restriction of this invariant
to involutions without fixed points (such an involution determines a chord
diagram) coincides with the value of the $\GL(m|n)$-weight system on this chord diagram.
We prove the recursion for $\GL(m|n)$-weight system, which is the same as the recursion for the $\GL$-one with $C_0=m-n$.

The paper is organized as follows.
In Sec.~\ref{sec:Lie}, we recall necessary information about Lie superalgebras,
especially, about the Lie superalgebra $\GL(m|n)$.
In Sec.~\ref{sec:cho}, we review the general definitions of Lie algebra/Lie superalgebra weight systems.
In Sec.~\ref{sec:rev}, we review the $\GL$-weight system and the recurrence rule we have introduced in our previous paper.
Sec.~\ref{sec:gen} is devoted to the definition of the extension of the $\GL(m|n)$-weight system to
permutations and calculations of some small examples.
In Sec.~\ref{sec:pro}, we prove another main theorem, which asserts that the $\GL(m|n)$-weight system is equivalent to the $\GL$-one.

The author is grateful to M.~Kazarian G.~I.~Olshanskii and A.~N.~Sergeev for valuable suggestions,
and to S.~Lando for permanent attention.

\section{Lie superalgebras}\label{sec:Lie}
First we recall the notion of Lie superalgebra, more details can be found in \cite{Kac}.
Everywhere in the paper, the ground field is $\mathbb{C}$, the field of complex numbers.

A {\em super vector space}, or a {\em $\BZ_2$-graded vector space}, is
a vector space decomposed as a direct sum $$V=V_0\oplus V_1.$$ The
indices (or {\em degrees}) 0 and 1 are thought of as elements of
$\BZ_2$; $V_0$ is called the {\em even} part of $V$, and $V_1$ is the
{\em odd} part of $V$. An element $x\in V$ is {\em homogeneous} if
it belongs to either $V_0$ or $V_1$. For $x$ homogeneous, we write
$|x|$ for the degree of $x$. The (super) dimension of $V$ is the
pair $(\dim V_0\, |\, \dim V_1)$ also sometimes written as $\dim V_0+ \dim
V_1$.

The vector space $\GL(V)$ of all endomorphisms of a super vector space $V$
is a super vector space itself: the subspace $\GL(V)_i$, $i=0,1$, consists of maps $f:V\to
V$ such that $f(V_j)\subseteq V_{j+i}$; each $f\in \GL(V)$ can be
written as a sum $f_0+f_1$ with $f_i\in \GL(V)_i$. If $V$ is
finite-dimensional, then the {\em supertrace} of $f$ is defined as
$${\rm sTr} f= \Tr f_0 - \Tr f_1.$$

A {\em superalgebra} is a super vector space $A$ together with a
bilinear product which respects the degree:
$$|xy|=|x|+|y|$$ for all homogeneous $x$ and $y$ in $A$. The {\em
supercommutator} in a superalgebra $A$ is a bilinear operation
defined on homogeneous $x,y\in A$ by
$$[x,y]=xy-(-1)^{|x|\,|y|}yx.$$ The elements of $A$ whose
supercommutator with the whole of $A$ is zero form the {\em super
center} of $A$.

The supercommutator satisfies the following identities:
$$|[x,y]|=|x|+|y|,$$
$$[x,y]=-(-1)^{|x|\,|y|}[y,x]$$
and
$$(-1)^{|z|\,|x|}[x,[y,z]]+(-1)^{|y|\,|z|}[z,[x,y]]+(-1)^{|x|\,|y|}[y,[z,x]]=0,$$
where $x,y,z$ are homogeneous. A super vector space with a bilinear
bracket $[\cdot,\cdot]$ satisfying these identities is called a {\em Lie
superalgebra}.\index{Lie superalgebra}

Each Lie superalgebra $\fg$ can be thought of as a subspace of its
{\em universal enveloping superalgebra} $U(\fg)$ defined as the
quotient of the tensor algebra $T(\fg)$ by the ideal $J$ generated by the elements of the form
$$x\otimes y- (-1)^{|x|\, |y|} y\otimes x-[x,y],$$
where $x$ and $y$ are arbitrary homogeneous elements of $\fg$.

Note that the tensor algebra $T(\fg)$ inherits
a $\BZ_2$-grading from $\fg$. Since the ideal $J$ is generated by homogeneous elements, it follows that $U(\fg)$ is also $\BZ_2$-graded and
the supercommutator in $U(\fg)$ induces the bracket in $\fg$.
We remark that  $ZU(\fg)$, the {\em center} of $U(\fg)$, can be defined as follows
\[
  ZU(\fg)=\{x\in U(\fg)|xy=yx \text{ for all } y\in \fg\}.
\]

The theory of Lie superalgebras was developed by V.~Kac \cite{Kac}; it closely parallels the usual Lie theory.

Let $\fg = \GL(m|n)$ be the vector space of all block square $(m+n)\times(m+n)$-matrices of the form
\begin{equation*}
X=\begin{bmatrix}A&B\\C&D\end{bmatrix},\label{eq:ma}
\end{equation*}
where $A$ is a square $m\times m$-matrix and $D$ is a square $n \times n$-matrix. Let $\fg_{0}$ denote the subspace
 of all such matrices with $B = C = 0$ and $\fg_{1}$ the subspace
  of all such matrices with $A = D = 0$. Then $\fg = \fg_{0}\oplus \fg_{1}$ is a $\BZ_2$-graded associative algebra with respect to ordinary matrix multiplication,
and $\fg_{i}$ is the set of all homogeneous elements of degree $i$, $i=0,1$. Elements of $\fg_{0}$ are called {\it even}, those of $\fg_{1}$ {\it odd}. Throughout what follows, if $|a|$ occurs in an expression, then it is assumed that $a$ is homogeneous,
so that $|a|=0$ provided~$a$ is even, and $|a|=1$ if~$a$  is odd.
The vector space $\GL(m|n)$ becomes a Lie superalgebra where the bracket is defined in terms of the usual matrix product by
\begin{equation*}
  [a,b]=ab-(-1)^{|a||b|} ba.
\end{equation*}

Similarly, if $V = V_0 \oplus V_1$ is a $\BZ_2$-graded vector space, then $\End(V)$ becomes a Lie superalgebra which we denote $\GL(V)$. If $\dim V_0 = m$ and $\dim V_1 = n$, then by choosing a homogeneous basis, we identify $\GL(V) \cong \GL(m|n)$.

If $X$ is as in (\ref{eq:ma}), we define the {\it supertrace} of $X$, denoted $\str(X)$, by
\begin{equation*}
\str(X)=\Tr(A)-\Tr(D)
\end{equation*}
and the bilinear form $\langle \cdot, \cdot\rangle$ on $\GL(m|n)$ is defined as $\langle a, b\rangle = \str(ab)$. Clearly, this bilinear form is nondegenerate.

We denote by $E_{ij}, i, j = 1,\dots,m + n$, the standard basis of the Lie superalgebra $\GL(m|n)$
consisting of matrix units.
The $\BZ_2$-grading on $\GL(m|n)$ is defined by $|E_{ij}|= \bar i +\bar j$,
where $\bar i$ is an element of $\BZ_2$ which equals $\bar 0$ or $\bar 1$
depending on whether $i \le m$ or $i > m$. The commutation relations in this basis are given by
\begin{equation*}
  [E_{ij},E_{kl}]=\delta_{jk}E_{il}-(-1)^{(\bar i +\bar j)(\bar k +\bar l)}\delta_{il}E_{kj}.
\end{equation*}

Below, we recall the result about the  images of the Casimir elements $C_k\in U(\GL(m|n))$
under the Harish-Chandra isomorphism related with supersymmetric functions.

\begin{definition}[supersymmetric functions]
The ring of {\em supersymmetric functions} $S(x_1,\dots,x_m|x_{m+1},\dots,x_{m+n})$ is defined as the subring of $\BC[x_1,\dots,x_{m+n}]$ generated by the homogeneous generators $p_k$ given by $$p_k=\sum _{i=1}^m x_i^k-(-1)^k \sum _{j=m+1}^{m+n} x_j^k.$$
\end{definition}

We have $f\in S(x_1,\dots,x_m|x_{m+1},\dots,x_{m+n})$ iff $f$ is symmetric separately in
$x_1,\dots,x_m$ and $x_{m+1},\dots,x_{m+n}$, and if substituting $x_m=t,x_{m+n}=-t$ in~$f$
provides a function independent of~$t$.

\begin{theorem}[Casimir elements $C_k$ in $U(\GL(m|n))$ \cite{Nwachuku}]\ \\
The center $ZU(\GL(m|n))$ of the universal enveloping algebra of the Lie superalgebra $\GL(m|n)$ is a polynomial algebra generated by the {\it Casimir elements} $C_k$, $k=1,2,\dots$, defined as
\[
  C_k=\sum_{i_1,\dots,i_k}^{m+n}(-1)^{\bar i_2+\bar i_3\dots+\bar i_k}E_{i_1i_2}E_{i_2i_3}\dots E_{i_{k-1}i_k}E_{i_ki_1},
\]
where we omit the tensor product sign $\otimes$ between the matrix units.
\end{theorem}

Below, we omit the tensor product sign provided this causes no confusion with the matrix product.

The triangular decomposition of the Lie superalgebra $\GL(m|n)$ gives a vector space decomposition
\[
  \fg=\fn_-\oplus\fh\oplus\fn_+
\]
where $\fn_-$ and $\fn_+$ are
the nilpotent subalgebras of, respectively, upper and lower triangular matrices in $\GL(m|n)$,
and $\fh$ is the subalgebra of diagonal matrices.

The universal enveloping algebra $U(\GL(m|n))$ of the Lie superalgebra $\GL(m|n)$ admits
the direct sum decomposition
\[
  U(\GL(m|n))=(\fn_-U(\GL(m|n))+U(\GL(m|n))\fn_+)\oplus U(\fh),
\]
The {\em Harish–Chandra projection} for $U(\GL(m|n))$ is the projection to the second summand
\[
  \phi:U(\GL(m|n))\to U(\fh)=\BC[E_{1,1},E_{2,2},\cdots,E_{m+n,m+n}],
\]
where $E_{1,1},E_{2,2},\cdots,E_{m+n,m+n}$ are the diagonal matrix units in $\GL(m|n)$; they commute with one another.

\begin{theorem}[a reformulation of Eq.~(3.4) in \cite{Nwachuku}]\label{th:Casimirs}
The Harish–Chandra projection, when restricted to the center $ZU(\GL(m|n))$,
is an algebra isomorphism to the polynomial algebra of supersymmetric functions $S(x_1,\dots,x_m|x_{m+1},\dots,x_{m+n})$ in the shifted generators $x_i=E_{ii}+r_i$, where
\[
  r_i=\sum _{j>i} (-1)^{\bar{i}+\bar{j}}+\frac{1}{2} \left(1-(-1)^{\bar{i}}\right).
\]
Explicitly, we have
\[
  1-\sum _{k=0}^{\infty } \phi(C_k) z^{k+1}=\prod _{i=1}^{m+n} \left(1-\frac{z}{1-z (-1)^{\bar{i}} x_i}\right){}^{(-1)^{\bar{i}}}
\]
\end{theorem}

\begin{remark}
  The center $ZU(\GL(m|n))$ of the universal enveloping algebra of the Lie superalgebra
  $\GL(m|n)$ is not a finitely generated polynomial algebra. However, for fixed $m,n$,
  one can express the variables $x_i$ in terms of the first $m+n$ Casimirs $C_1,\dots,C_{m+n}$.
  After substituting these expressions in the formula above,  one can write the higher Casimirs $C_k$,
 for $k>m+n$, as rational functions in the variables $C_1,\dots,C_{m+n}$.
\end{remark}

\section{Chord diagrams and weight systems}\label{sec:cho}

Below, we use standard notions from the theory of finite order knot invariants;
see, e.g.~\cite{Chmutov2012,lando2013graphs}.

A {\it chord diagram\/} of order~$n$ is an oriented circle (called the \emph{Wilson loop}) endowed with~$2n$
pairwise distinct points split into~$n$ disjoint pairs, considered up to
orientation-preserving diffeomorphisms of the circle.

A~{\it weight system\/} is a function $w$ on chord diagrams satisfying the $4$-term
relations; see Fig.~\ref{fourtermrelation}.

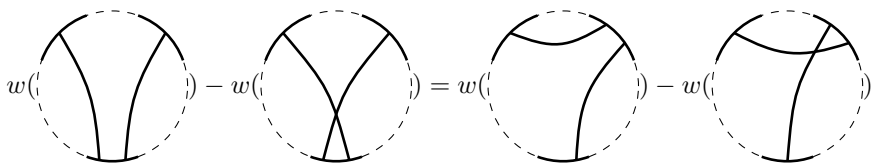
\begin{figure}[ht]
\[
w(\begin{tikzpicture}[baseline={([yshift=-.5ex]current bounding box.center)}]
  \draw[dashed] (0,0) circle (1);
  \draw[line width=1pt]  ([shift=( 20:1cm)]0,0) arc [start angle= 20, end angle= 70, radius=1];
  \draw[line width=1pt]  ([shift=(110:1cm)]0,0) arc [start angle=110, end angle=160, radius=1];
  \draw[line width=1pt]  ([shift=(250:1cm)]0,0) arc [start angle=250, end angle=290, radius=1];
  \draw[line width=1pt] (45:1) ..  controls (5:0.3) and (-40:0.3)  .. (280:1);
  \draw[line width=1pt] (135:1) ..  controls (175:0.3) and (220:0.3)  .. (260:1);
\end{tikzpicture})  -
w(\begin{tikzpicture}[baseline={([yshift=-.5ex]current bounding box.center)}]
  \draw[dashed] (0,0) circle (1);
  \draw[line width=1pt]  ([shift=( 20:1cm)]0,0) arc [start angle= 20, end angle= 70, radius=1];
  \draw[line width=1pt]  ([shift=(110:1cm)]0,0) arc [start angle=110, end angle=160, radius=1];
  \draw[line width=1pt]  ([shift=(250:1cm)]0,0) arc [start angle=250, end angle=290, radius=1];
  \draw[line width=1pt] (45:1) ..  controls (-5:0.1) and (-50:0.1)  .. (260:1);
  \draw[line width=1pt] (135:1) ..  controls (185:0.1) and (225:0.1)  .. (280:1);
\end{tikzpicture})  =
w(\begin{tikzpicture}[baseline={([yshift=-.5ex]current bounding box.center)}]
  \draw[dashed] (0,0) circle (1);
  \draw[line width=1pt]  ([shift=( 20:1cm)]0,0) arc [start angle= 20, end angle= 70, radius=1];
  \draw[line width=1pt]  ([shift=(110:1cm)]0,0) arc [start angle=110, end angle=160, radius=1];
  \draw[line width=1pt]  ([shift=(250:1cm)]0,0) arc [start angle=250, end angle=290, radius=1];
  \draw[line width=1pt] (35:1) ..  controls (0:0.3) and (-45:0.3)  .. (280:1);
  \draw[line width=1pt] (135:1) ..  controls (105:0.5) and (85:0.5)  .. (55:1);
\end{tikzpicture})  -
w(\begin{tikzpicture}[baseline={([yshift=-.5ex]current bounding box.center)}]
  \draw[dashed] (0,0) circle (1);
  \draw[line width=1pt]  ([shift=( 20:1cm)]0,0) arc [start angle= 20, end angle= 70, radius=1];
  \draw[line width=1pt]  ([shift=(110:1cm)]0,0) arc [start angle=110, end angle=160, radius=1];
  \draw[line width=1pt]  ([shift=(250:1cm)]0,0) arc [start angle=250, end angle=290, radius=1];
  \draw[line width=1pt] (55:1) ..  controls (5:0.1) and (-40:0.1)  .. (270:1);
  \draw[line width=1pt] (135:1) ..  controls (105:0.4) and (65:0.4)  .. (35:1);
\end{tikzpicture})
\]
\caption{$4$-term relations}
\label{fourtermrelation}
\end{figure}

In figures, the outer circle of the chord diagram is always assumed to be
oriented counterclockwise. Dashed arcs may contain ends of arbitrary sets
of chords, same for all the four terms in the picture.
Let us recall the construction of the $\GL(N)$-weight system for permutations, as introduced in \cite{ZY2}.

Given a Lie algebra $\fg$ equipped with a non-degenerate invariant bilinear form, one can construct a weight system with values in the center of its universal enveloping algebra $U(\fg)$.
This is the form M. Kontsevich~\cite{kon1993} gave to a construction due to D. Bar-Natan~\cite{bar1995vassiliev}.
Kontsevich’s construction proceeds as follows.

Given a chord diagram $D$ with $n$ chords, we first choose a base point on the circle, away from the ends of the chords of $D$. This gives a linear order on the endpoints of the chords, increasing in the positive direction of the Wilson loop.
We order the chords of $D$ according to the order of their left endpoints. Let us number the chords from $1$ to $n$, and their
endpoints from $1$ to $2n$, in the increasing order. Then $D$ gives a permutation $\sigma_D$ of the
set $\{1,2,\dots,2n\}$ as follows. For $1 \le i \le n$ the
permutation $\sigma_D$ sends $2i-1$ to the (number of the) left endpoint of the $i$th
chord, and $2i$ to the (number of the) right endpoint of the same chord.
The permutation $\sigma_D$ is exactly the re-arrangement, which sends the endpoints of the diagram with $n$
consecutive isolated chords into $D$.

\begin{definition}[Universal Lie algebra weight system]
Let $\fg$ be a metrized Lie algebra over $\mathbb{C}$, that is, a Lie algebra with an ad-invariant non-degenerate bilinear form $\langle\cdot,\cdot\rangle$.
The bilinear form $\langle\cdot,\cdot\rangle$ on $\fg$ is a tensor in $\fg^* \otimes \fg^*$. The algebra $\fg$ being
metrized,  we can identify $\fg^*$ with $\fg$ and think of $\langle\cdot,\cdot\rangle$ as of an element of
$\fg\otimes\fg$. The permutation $\sigma_D$ acts on $\fg^{\otimes2n}$ by interchanging the factors. The
value of the universal Lie algebra weight system $w_\fg (D)$ is the image of
the $n$th tensor power $\langle\cdot,\cdot\rangle^{\otimes n}$ under the map
\[
  \fg^{\otimes2n} \xrightarrow{\sigma_D} \fg^{\otimes2n}\xrightarrow{\phantom{\sigma_D}} U(\fg),
\]
where the second map is the restriction of the natural projection of the tensor algebra on $\fg$
to its universal enveloping algebra.
\end{definition}

This construction of Lie algebra weight systems works also for Lie superalgebras, which are more general than Lie algebras.
Let us recall the definition of the Lie superalgebra weight system on chord diagrams.

If $\fg$ is a metrized Lie superalgebra, the very same construction works
with only one modification: re-arranging the factors in the final step should
be done with certain care. Instead of simply permuting the factors in the
tensor product one should use a representation of the symmetric
group $S_k$ on $k$ letters, which acts on the $k$~th tensor power of any super
vector space.

This representation is defined as follows. Define
\begin{align*}
  S: \fg\otimes\fg&\to\fg\otimes\fg; \\
  S:x\otimes y&\mapsto (-1)^{|x||y|} y\otimes x.
\end{align*}
The map $S$ is an involution; in other words, it defines a representation of the symmetric group $S_2$ on the vector space $\fg^{\otimes2}$. More generally, the representation of $S_k$ on $\fg^{\otimes k}$ is defined by sending the elementary transposition $(i,i + 1)$ to $id^{\otimes i-1} \otimes S \otimes id^{\otimes k-i-1} $, $i=1,\dots,k-1$.

\begin{definition}[Universal Lie superalgebra weight system]
Let $\fg$ be a metrized Lie superalgebra over $\mathbb{C}$, that is, a Lie superalgebra with an ad-invariant non-degenerate bilinear form $\langle\cdot,\cdot\rangle$.
The bilinear form $\langle\cdot,\cdot\rangle$ on $\fg$ is a tensor in $\fg^* \otimes \fg^*$. The algebra $\fg$ being
metrized, we can identify $\fg^*$ with $\fg$ and think of $\langle\cdot,\cdot\rangle$ as of an element of
$\fg\otimes\fg$. The permutation $\sigma_D$ acts on $\fg^{\otimes2n}$ by interchanging the factors. The
value of the universal Lie superalgebra weight system $w_\fg (D)$ is the image of
the $n$th tensor power $\langle\cdot,\cdot\rangle^{\otimes n}$ under the map
\[
  \fg^{\otimes2n} \xrightarrow{S_{2n}\circ\sigma_D} \fg^{\otimes2n}\xrightarrow{\phantom{S_{2n}\circ\sigma_D}} U(\fg),
\]
where the second map is the restriction of the natural projection of the tensor algebra on $\fg$
to its universal enveloping algebra.
\end{definition}

\begin{claim}~{\rm\cite{Chmutov2012,kon1993}} The function $w_\fg :D\mapsto w_\fg (D)$ on chord diagrams has the following properties:
\begin{enumerate}
\item the element $w_\fg (D)$ does not depend on the choice of the base point on the diagram;
\item its image belongs to the ad-invariant subspace
\[
  U(\fg)^{\fg}=\{x\in U(\fg)|xy=yx \text{ for all } y\in \fg\}=ZU(\fg);
\]
\item this map from chord diagrams to $ZU(\fg)$ satisfies the 4-term relations.
\end{enumerate}
\end{claim}
Therefore, $w_\fg$ is a weight system taking values in $ZU(\fg)$.

\section{Review of the $\GL$-weight system and the Recurrence Rule}\label{sec:rev}
Let us recall the construction of the $\GL(N)$-weight system for permutations, as introduced in \cite{ZY2}.
For a permutation $\sigma\in S_k$, set

\[
  w_{\GL(N)}(\sigma)=\sum_{i_1,\cdots,i_k=1}^N E_{i_1i_{\sigma(1)}}E_{i_2i_{\sigma(2)}}\cdots E_{i_ki_{\sigma(k)}}\in U(\GL(N)).
\]

For example, the standard Casimir generator
$$C_k=\sum^N_{i_1,\cdots,i_k=1}E_{i_1i_2}E_{i_2i_3}\cdots E_{i_{k-1}i_km}E_{i_ki_1}$$
corresponds to the cyclic permutation $1\mapsto2\mapsto\cdots\mapsto k\mapsto1\in S_k$.

It is shown in \cite{ZY2} that
\begin{itemize}
  \item the value of $w_{\GL(N)}$ on any permutation lies in the center of $U(\GL(N))$;
  \item this element is invariant under conjugation by  the standard cyclic permutation, that is

    $w_{\GL(N)}(\sigma)=\sum_{i_1,\cdots,i_k=1}^N E_{i_2i_{\sigma(2)}}\cdots E_{i_ki_{\sigma(k)}}E_{i_1i_{\sigma(1)}}$.
\end{itemize}

\begin{definition}[digraph of a permutation]
  Let us represent  a permutation as an oriented graph.
The $k$ vertices of the graph correspond to the permuted elements.
They are ordered and are placed on a horizontal arrow looking right. The arc arrows
show the action of the permutation (so that each vertex is incident with exactly
one incoming and one outgoing arc edge).
  The digraph of a permutation $\sigma\in S_k$
consists of these~$k$ vertices and $k$ oriented edges, for example:
  \[
  (\left(1\ n+1)(2\ n+2)\cdots(n\ 2n)\right)=\begin{tikzpicture}[baseline={([yshift=-.5ex]current bounding box.center)},decoration={markings, mark= at position .55 with {\arrow{stealth}}}]
    \draw[->,thick] (-2,0)--(2,0);
    \draw[blue,postaction={decorate}] (-1.8,0) ..  controls (-1,.5) ..(.2,0);
    \draw[blue,postaction={decorate}] (-1.4,0) ..  controls (-.6,.5) ..(.6,0);
    \draw[blue,postaction={decorate}] (-.2,0) ..  controls (1,.5) ..(1.8,0);
    \draw[blue,postaction={decorate}] (.2,0) ..  controls (-1,-.5) ..(-1.8,0);
    \draw[blue,postaction={decorate}] (.6,0) ..  controls (-.6,-.5) ..(-1.4,0);
    \draw[blue,postaction={decorate}] (1.8,0) ..  controls (1,-.5) ..(-.2,0);
    \fill[black] (-1.8,0) circle (1pt) node[below] {\tiny 1};
    \fill[black] (-1.4,0) circle (1pt) node[below] {\tiny 2};
    \fill[black] (-0.2,0) circle (1pt) node[below] {\tiny n};
    \fill[black] ( .15,0) circle (1pt) node[below] {\tiny n+1};
    \fill[black] ( .6,0) circle (1pt)  node[below]{\tiny n+2};
    \fill[black] ( 1.8,0) circle (1pt) node[below] {\tiny 2n};
    \node[below] at (-.8,0) {$\cdots$};
    \node[below] at (1.2,0) {$\cdots$};
\end{tikzpicture}
  \]
\end{definition}

\begin{theorem}[\cite{ZY2}]
The value of the $w_{\GL(N)}$ invariant of permutations possesses the following properties:
\begin{itemize}
  \item for the empty graph (with no vertices) the value of $w_{\GL(N)}$ is equal to $1$,
    $w_{\GL(N)}(\textcircled{})=1$;
  \item $w_{\GL(N)}$ is multiplicative with respect to concatenation of permutations;
  \item for a cyclic permutation {\rm(}with the cyclic order on the set of permuted elements  compatible with the permutation{\rm)},
 the value of $w_{\GL(N)}$ is the standard generator,
    $w_{\GL(N)}(1\mapsto2\mapsto\cdots\mapsto k\mapsto1)=C_k$.

\item {\rm(}\textbf{Recurrence Rule}{\rm)} For the graph of an arbitrary permutation $\sigma$ in~$S_k$,
and for any two neighboring elements $l,l+1$, of the permuted set $\{1,2,\dots,k\}$, we have
for the values of the $w_{\GL(N)}$ weight system
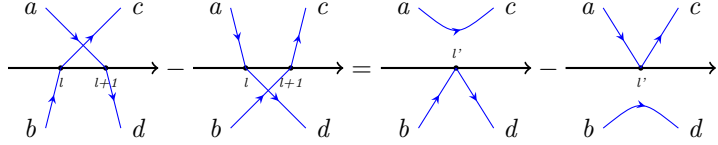
\begin{equation*}
  \begin{tikzpicture}[baseline={([yshift=-.5ex]current bounding box.center)},decoration={markings, mark= at position .55 with {\arrow{stealth}}}]
    \draw[->,thick] (-1,0) --  (1,0);
    \fill[black] (-.3,0) circle (1pt) node[below] {\tiny l};
    \fill[black] (.3,0) circle (1pt) node[below] {\tiny l+1};
    \draw (-.5,.8) node[left] {a};
    \draw (-.5,-.8) node[left] {b};
    \draw (.5,.8) node[right] {c};
    \draw (.5,-.8) node[right] {d};
    \draw[blue,postaction={decorate}] (-.5,.8) -- (.3,0);
    \draw[blue,postaction={decorate}] (-.3,0) -- (.5,.8);
    \draw[blue,postaction={decorate}] (-.5,-.8) -- (-.3,0);
    \draw[blue,postaction={decorate}] (.3,0) -- (.5,-.8);
    %\draw (0,-1.2) node { $\sigma$};
  \end{tikzpicture}-
  \begin{tikzpicture}[baseline={([yshift=-.5ex]current bounding box.center)},decoration={markings, mark= at position .55 with {\arrow{stealth}}}]
    \draw[->,thick] (-1,0) --  (1,0);
    \fill[black] (.3,0) circle (1pt) node[below] {\tiny l+1};
    \fill[black] (-.3,0) circle (1pt) node[below] {\tiny l};
    \draw (-.5,.8) node[left] {a};
    \draw (-.5,-.8) node[left] {b};
    \draw (.5,.8) node[right] {c};
    \draw (.5,-.8) node[right] {d};
    \draw[blue,postaction={decorate}] (-.5,.8) -- (-.3,0);
    \draw[blue,postaction={decorate}] (.3,0) -- (.5,.8);
    \draw[blue,postaction={decorate}] (-.5,-.8) -- (.3,0);
    \draw[blue,postaction={decorate}] (-.3,0) -- (.5,-.8);
    %\draw (0,-1.2) node { $\sigma^{*}$};
  \end{tikzpicture}=
  \begin{tikzpicture}[baseline={([yshift=-.5ex]current bounding box.center)},decoration={markings, mark= at position .55 with {\arrow{stealth}}}]
    \draw[->,thick] (-1,0)  -- (1,0);
    \fill[black] (0,0) circle (1pt) node[above] {\tiny l'};
    \draw (-.5,.8) node[left] {a};
    \draw (-.5,-.8) node[left] {b};
    \draw (.5,.8) node[right] {c};
    \draw (.5,-.8) node[right] {d};
    \draw[blue,postaction={decorate}] (-.5,.8) ..controls (0,.4) .. (.5,.8);
    \draw[blue,postaction={decorate}] (-.5,-.8) -- (0,0);
    \draw[blue,postaction={decorate}] (0,0) -- (.5,-.8);
    %\draw (0,-1.2) node { $\sigma'$};
  \end{tikzpicture}-
  \begin{tikzpicture}[baseline={([yshift=-.5ex]current bounding box.center)},decoration={markings, mark= at position .55 with {\arrow{stealth}}}]
    \draw[->,thick] (-1,0) --  (1,0);
    \fill[black] (0,0) circle (1pt) node[below] {\tiny l'};
    \draw (-.5,.8) node[left] {a};
    \draw (-.5,-.8) node[left] {b};
    \draw (.5,.8) node[right] {c};
    \draw (.5,-.8) node[right] {d};
    \draw[blue,postaction={decorate}] (-.5,-.8) ..controls (0,-.4) .. (.5,-.8);
    \draw[blue,postaction={decorate}] (-.5,.8) -- (0,0);
    \draw[blue,postaction={decorate}] (0,0) -- (.5,.8);
    %\draw (0,-1.2) node { $\sigma''$};
  \end{tikzpicture}
\end{equation*}

In the diagrams on the left, two horizontally neighboring vertices
and the edges incident to them are depicted, while
on the right these two vertices are replaced with a single one;
the other vertices are placed somewhere on the circle and their positions are the same on all
diagrams participating in the relations, but the numbers of the vertices to the right
of the latter are to be decreased by~$1$.

For the special case $\sigma(l+1)=l$, the recurrence looks like follows:
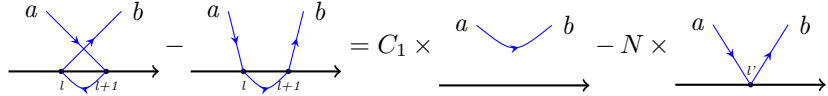
\begin{equation*}
  \begin{tikzpicture}[baseline={([yshift=-.5ex]current bounding box.center)},decoration={markings, mark= at position .55 with {\arrow{stealth}}}]
    \draw[->,thick] (-1,0) --  (1,0);
    \fill[black] (-.3,0) circle (1pt) node[below] {\tiny l};
    \fill[black] (.3,0) circle (1pt) node[below] {\tiny l+1};
    \draw (-.5,.8) node[left] {a};
    \draw (.5,.8) node[right] {b};
    \draw[blue,postaction={decorate}] (-.5,.8) -- (.3,0);
    \draw[blue,postaction={decorate}] (-.3,0) -- (.5,.8);
    \draw[blue,postaction={decorate}] (.3,0) ..controls(0,-.3).. (-.3,0);
    %\draw (0,-0.7) node { $\sigma$};
  \end{tikzpicture}-
  \begin{tikzpicture}[baseline={([yshift=-.5ex]current bounding box.center)},decoration={markings, mark= at position .55 with {\arrow{stealth}}}]
    \draw[->,thick] (-1,0) --  (1,0);
    \fill[black] (.3,0) circle (1pt) node[below] {\tiny l+1};
    \fill[black] (-.3,0) circle (1pt) node[below] {\tiny l};
    \draw (-.5,.8) node[left] {a};
    \draw (.5,.8) node[right] {b};
    \draw[blue,postaction={decorate}] (-.5,.8) -- (-.3,0);
    \draw[blue,postaction={decorate}] (.3,0) -- (.5,.8);
    \draw[blue,postaction={decorate}] (-.3,0) ..controls(0,-.3).. (.3,0);
    %\draw (0,-0.7) node { $\sigma*$};
  \end{tikzpicture}=C_1\times
  \begin{tikzpicture}[baseline={([yshift=-.5ex]current bounding box.center)},decoration={markings, mark= at position .55 with {\arrow{stealth}}}]
    \draw[->,thick] (-1,0)  -- (1,0);
    \draw (-.5,.8) node[left] {a};
    \draw (.5,.8) node[right] {b};
    \draw[blue,postaction={decorate}] (-.5,.8) ..controls (0,.4) .. (.5,.8);
    %\draw (0,-0.4) node { $\sigma'$};
  \end{tikzpicture}-N\times
  \begin{tikzpicture}[baseline={([yshift=-.5ex]current bounding box.center)},decoration={markings, mark= at position .55 with {\arrow{stealth}}}]
    \draw[->,thick] (-1,0) --  (1,0);
    \fill[black] (0,0) circle (1pt) node[above] {\tiny l'};
    \draw (-.5,.8) node[left] {a};
    \draw (.5,.8) node[right] {b};
    \draw[blue,postaction={decorate}] (-.5,.8) -- (0,0);
    \draw[blue,postaction={decorate}] (0,0) -- (.5,.8);
    %\draw (0,-0.4) node { $\sigma''$};
  \end{tikzpicture}
\end{equation*}
These relations are indeed a recursion, that is, they allow one to
replace the computation of $w_{\GL(N)}$ on a permutation with its computation on simpler permutations.
\end{itemize}

\end{theorem}

The recursion rule of the theorem defines a weight system on permutations taking
values in the ring of polynomials in infinitely many variables ${\mathbb C}[C_0=N,C_1,C_2,\dots]$.
We denote this universal $\GL$-weight system by $w_\GL$, where $C_0$ coincides with the number $N$ 
in the second recursion rule and $C_k$ corresponds to the standard length $k$ cyclic permutation. The recursion in the theorem allows one
to compute this weight system effectively.

\section{Extension of the $\GL(m|n)$-weight system to permutations}\label{sec:gen}

We define   $w_{\GL(m|n)}$ on permutations
in the following way, which is similar to the definition for $w_{\GL(N)}$.

For a permutation $\sigma\in S_k$, set
\[
  w_{\GL(m|n)}(\sigma)=\sum_{i_1,\cdots,i_k=1}^{m+n} (-1)^{f_\sigma} E_{i_1i_{\sigma(1)}}E_{i_2i_{\sigma(2)}}\cdots E_{i_ki_{\sigma(k)}},
\]
where $f_\sigma$ is the sign function which is a polynomial in $\bar i_1,\bar i_2,\dots,\bar i_k$ in the field $\BZ_2$ defined below.

The sign function $f_\sigma$ is a polynomial that has linear and quadratic terms only. For example, for the standard cyclic permutation $(1 2 \dots k):1\to2\to\dots\to k\to 1$, we have $f_{(1 2 \dots k)}=\bar i_2+\dots+\bar i_k$.

We say that an index $a$, $1\le a\le k$, is {\it distinguished} with respect to~$\sigma$
if $\sigma^{-1}(a)<a$. The set of distinguished indices is denoted by $P_1(\sigma)\subset \{1,\dots,k\}$.
We say that a pair of indices $(a,b)$, $1\le a<b\le k$, is {\it distinguished} if the two pairs of distinct real numbers $(\sigma^{-1}(a)+\epsilon,a-\epsilon)$ and $(\sigma^{-1}(b)+\epsilon,b-\epsilon)$ alternate; here $\epsilon>0$ is a small real number, say, $\epsilon=\frac{1}{3}$. The set of distinguished pairs of indices is denoted by $P_2(\sigma)\subset\{1,\dots,k\}\times\{1,\dots,k\}$.

\begin{definition}
The sign function $f_\sigma$ of a permutation $\sigma\in S_k$ is defined by
 \[
  f_\sigma(\bar i_1,\bar i_2, \dots)=\sum_{a\in P_1(\sigma)}\bar i_a+\sum_{(a,b)\in P_2(\sigma)}\bar i_a\bar i_b.
 \]
\end{definition}

A more convenient treatment of the invariant $w_{\GL(m|n)}(\sigma)$ and the sign function uses the language of digraphs from the previous section.

The set of indices participating in the summation will be labelled by the edges (rather than by vertices). For each vertex $i$, we denote by $in(i)$ and $out(i)$ the incoming edge and outcoming edge incident to the vertex $i$, respectively. With this notation, we have
\[
  w_{\GL(m|n)}(\sigma)=\sum_{i_1,\cdots,i_k=1}^{m+n}  (-1)^{f_\sigma} E_{i_{in(1)}i_{out(1)}}\cdots E_{i_{in(k)}i_{out(k)}}.
\]
The original formula corresponds to the numbering of the edges such
that the edge $i\to j$ is numbered $j$. The result is obviously independent of the numbering.

With this notation, an edge is distinguished if it is directed from left to right. A pair of edges with pairwise distinct ends is distinguished  if the corresponding pairs of vertices alternate. If the edges have common vertices, we first bring them to a general position by shifting slightly the beginning of each edge to the right and the endpoint of each edge to the left, and then check whether the pairs of ends of the shifted edges do alternate.

\begin{claim}
  For the cyclic permutation $\sigma=(1 2 \cdots k)$, the diagram is $\begin{tikzpicture}[baseline={([yshift=-.5ex]current bounding box.center)},decoration={markings, mark= at position .7 with {\arrow{stealth}}}]
    \draw[->,thick] (-1,0)--(1,0);
    \draw[blue,postaction={decorate}] (-.9,0) ..  controls (-.75,.2) ..(-.6,0);
    \draw[blue,postaction={decorate}] (-.6,0) ..  controls (-.45,.2) ..(-.3,0);
    \draw[blue,postaction={decorate}] (.3,0) ..  controls (.45,.2) ..(.6,0);
    \draw[blue,postaction={decorate}] (.6,0) ..  controls (.75,.2) ..(.9,0);
    \draw[blue,postaction={decorate}] (.9,0) ..  controls (0,-.5) ..(-.9,0);
    \fill[black] (-.9,0) circle (1pt) node[below] {\tiny 1};
    \fill[black] (-.6,0) circle (1pt) node[below] {\tiny 2};
    \fill[black] (-.3,0) circle (1pt) node[below] {\tiny 3};
    \fill[black] ( .3,0) circle (1pt) node[below] {\tiny k-2};
    \fill[black] ( .6,0) circle (1pt)  node[below]{\tiny k-1};
    \fill[black] ( .9,0) circle (1pt) node[below] {\tiny k};
    \node[below] at (0,.3) {$\cdots$};
\end{tikzpicture}$;
and we have $f_{(1 2 \cdots k)}=\bar i_2+\dots+\bar i_k$.
\end{claim}

Assume that two permutations $\sigma$ and $\sigma'$ are conjugate by a transposition of two neighboring elements. Then these two elements are the endpoints of the four edges $a,b,c,d$ as shown in the picture below (among the edges $a,b,c,d$ there could be pairs of coincident ones).

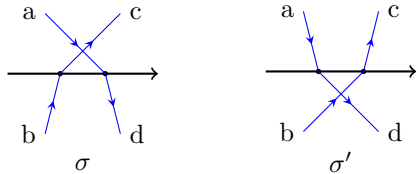
\[
  \begin{tikzpicture}[baseline={([yshift=-.5ex]current bounding box.center)},decoration={markings, mark= at position .55 with {\arrow{stealth}}}]
    \draw[->,thick] (-1,0) --  (1,0);
    \fill[black] (-.3,0) circle (1pt) ;
    \fill[black] (.3,0) circle (1pt) ;
    \draw (-.5,.8) node[left] {a};
    \draw (-.5,-.8) node[left] {b};
    \draw (.5,.8) node[right] {c};
    \draw (.5,-.8) node[right] {d};
    \draw[blue,postaction={decorate}] (-.5,.8) -- (.3,0);
    \draw[blue,postaction={decorate}] (-.3,0) -- (.5,.8);
    \draw[blue,postaction={decorate}] (-.5,-.8) -- (-.3,0);
    \draw[blue,postaction={decorate}] (.3,0) -- (.5,-.8);
    \draw (0,-1.2) node { $\sigma$};
   \end{tikzpicture} \ \ \ \ \ \ \ \ \ \ \ \
  \begin{tikzpicture}[baseline={([yshift=-.5ex]current bounding box.center)},decoration={markings, mark= at position .55 with {\arrow{stealth}}}]
    \draw[->,thick] (-1,0) --  (1,0);
    \fill[black] (.3,0) circle (1pt) ;
    \fill[black] (-.3,0) circle (1pt) ;
    \draw (-.5,.8) node[left] {a};
    \draw (-.5,-.8) node[left] {b};
    \draw (.5,.8) node[right] {c};
    \draw (.5,-.8) node[right] {d};
    \draw[blue,postaction={decorate}] (-.5,.8) -- (-.3,0);
    \draw[blue,postaction={decorate}] (.3,0) -- (.5,.8);
    \draw[blue,postaction={decorate}] (-.5,-.8) -- (.3,0);
    \draw[blue,postaction={decorate}] (-.3,0) -- (.5,-.8);
    \draw (0,-1.2) node { $\sigma'$};
  \end{tikzpicture}
\]
\begin{lemma}\label{le:1}
  The sign functions $f_\sigma$ and $f_{\sigma'}$ are related by
  \[
  f_{\sigma'}=f_\sigma +(\bar i_a+\bar i_d)(\bar i_b+ \bar i_c).
  \]
  In other words, each of the four pairs of edges $(a,c),(a,d),(b,c),(b,d)$ changes the property of being distinguished when one passes from the permutation $\sigma$ to $\sigma'$.
\end{lemma}

Since the sign function $f_\sigma$ matches the sign in the Casimir elements and this lemma says the sign function $f_\sigma$ matches the involution operation $S$, we have
\begin{claim}
   The $\GL(m|n)$-weight system for chord diagrams in~\cite{Va,FKV} is a special case of the $\GL(m|n)$-weight system for permutations,
    where we treat a chord diagram with $k$ chords as an involution without fixed points on the set of $2k$ elements.
\end{claim}

\begin{example}
  Let $\sigma=(132)\in S_3$. According to the definition,
  \begin{align*}
    w_{\GL(m|n)}((132))&=\sum_{i_1,i_2,i_3=1}^{m+n} (-1)^{f_{(132)}} E_{i_1i_3}E_{i_2i_1} E_{i_3i_2},
      \intertext{where } f_{(132)}&=\bar i_3+\bar i_1\bar i_3+\bar i_3\bar i_2+\bar i_1\bar i_2,
      \intertext{and we have the Lie superbracket $[E_{ij},E_{kl}]=\delta_{jk}E_{il}-(-1)^{(\bar i +\bar j)(\bar k +\bar l)}\delta_{il}E_{kj}$.}
      \intertext{Now,}C_3-w_{\GL(m|n)}((132))&=\sum_{i_1,i_2,i_3=1}^{m+n} (-1)^{\bar i_3+\bar i_2} E_{i_1i_3}\left(E_{i_3i_2} E_{i_2i_1}-(-1)^{(\bar i_3+\bar i_2)(\bar i_1+\bar i_2)}  E_{i_3i_2} E_{i_2i_1} \right)\\
                                              &=\sum_{i_1,i_2,i_3=1}^{m+n} (-1)^{\bar i_3+\bar i_2} E_{i_1i_3}\left[E_{i_3i_2}, E_{i_2i_1}\right]\\
                                              &=\sum_{i_1,i_2,i_3=1}^{m+n} (-1)^{\bar i_3+\bar i_2} E_{i_1i_3}(\delta_{i_2i_2}E_{i_3i_1}-\delta_{i_3i_1} (-1)^{(i_3+\bar i_2)(i_1+\bar i_2)} E_{i_2i_2})\\
                                              &=(m-n)\sum_{i_1,i_3=1}^{m+n} (-1)^{\bar i_3} E_{i_1i_3}E_{i_3i_1}-\sum_{i_1,i_2=1}^{m+n}  E_{i_1i_1} E_{i_2i_2}\\
                                              &=(m-n)C_2-C_1^2
  \end{align*}
  Finally, we get $w_{\GL(m|n)}((132))=C_3-(m-n)C_2+C_1^2$.
\end{example}

\newpage
\begin{result} \ \\
  \begin{tabular}{c|c|c|c}
  k&$\sigma$&$f_\sigma$&polynomial in Casimir elements\\ \hline
  \multirow{2}*{2}  & Id&$0$ & $C_1^2$\\
    &(1 2)& $\bar{i}_2$& $C_2$\\\hline
    \multirow{6}*{3}&Id& 0 &$C_1^3$ \\
    &(1 2)& $\bar{i}_2$& $C_1C_2$\\
    &(2 3)& $\bar{i}_3$& $C_1C_2$\\
    &(1 3)& $\bar{i}_3$& $C_1C_2$\\
    &(1 2 3)& $\bar{i}_2+\bar{i}_3$& $C_3$\\
    &(1 3 2)& $\bar i_3+\bar i_2+(\bar i_3+\bar i_2)(\bar i_1+\bar i_2)$& $C_3-(m-n)C_2+C_1^2$\\\hline
    \multirow{25}*{4}&Id& 0 &$C_1^4$ \\
    &(1 2)& $\bar{i}_2$& $C_1^2C_2$\\
    &(2 3)& $\bar{i}_3$& $C_1^2C_2$\\
    &(1 3)& $\bar{i}_3$& $C_1^2C_2$\\
    &(1 4)& $\bar{i}_4$& $C_1^2C_2$\\
    &(2 4)& $\bar{i}_4$& $C_1^2C_2$\\
    &(3 4)& $\bar{i}_4$& $C_1^2C_2$\\
    &(1 2)(3 4)&$\bar i_2+\bar{i}_4$& $C_2^2$\\
    &(1 4)(2 3)&$\bar i_3+\bar{i}_4$& $C_2^2$\\
    &(1 3)(2 4)&$\bar i_3+\bar{i}_4+(\bar i_2+\bar{i}_4)(\bar i_1+\bar{i}_3)$& $C_2^2-(m-n)C_2+C_1^2$\\
    &(1 2 3)& $\bar{i}_2+\bar{i}_3$& $C_1C_3$\\
    &(1 2 4)& $\bar{i}_2+\bar{i}_4$& $C_1C_3$\\
    &(1 3 4)& $\bar{i}_3+\bar{i}_4$& $C_1C_3$\\
    &(2 3 4)& $\bar{i}_3+\bar{i}_4$& $C_1C_3$\\
    &(1 3 2)& $(\bar i_2+\bar i_3)(\bar i_1+\bar i_3)$& $C_1(C_3-(m-n)C_2+C_1^2)$\\
    &(1 4 2)& $(\bar i_2+\bar i_4)(\bar i_1+\bar i_4)$& $C_1(C_3-(m-n)C_2+C_1^2)$\\
    &(1 4 3)& $(\bar i_3+\bar i_4)(\bar i_1+\bar i_4)$& $C_1(C_3-(m-n)C_2+C_1^2)$\\
    &(2 4 3)& $(\bar i_2+\bar i_4)(\bar i_3+\bar i_4)$& $C_1(C_3-(m-n)C_2+C_1^2)$\\
    &(1 2 3 4)& $\bar i_2+\bar i_3+\bar i_4$& $C_4$\\
    &(1 2 4 3)& $\bar i_2+\bar i_4+\bar i_1\bar i_4+\bar i_1\bar i_3+\bar i_4\bar i_3$& $C_4-(m-n)C_3+C_1C_2$\\
    &(1 3 2 4)& & $C_4-(m-n)C_3+C_1C_2$\\
    &(1 3 4 2)& & $C_4-(m-n)C_3+C_1C_2$\\
    &(1 4 2 3)& & $C_4-(m-n)C_3+C_1C_2$\\
    &\multirow{2}*{(1 4 3 2)}&\multirow{2}*{$\bar i_4+(\bar i_1+\bar i_3)(\bar i_4+\bar i_2)$} & $C_4-2(m-n)C_3+2C_1C_2+$\\
    &&&$+(m-n)^2C_2-(m-n)C_1$\\
  \end{tabular}
\end{result}

In all the above examples,
the value of the $\GL(m|n)$-weight system is a polynomial in the difference $m-n$.
The following stronger theorem, which is another main result of the present paper, asserts that this is always true.

\begin{theorem}\label{the:main}
  The weight system $w_{\GL(m|n)}$ for permutations is the result of substituting $m-n$ for~$C_0$, and
  the $k$~th Casimir element in $\GL(m|n)$ for $C_k$, $k>0$, in the weight system $w_{\GL}$.
\end{theorem}

The proof of this theorem is given in the next section.

\begin{example}
In~\cite{FKV}, a recurrence relation for computing the values of the $\GL(1|1)$-weight
system is given. Our approach suggests another recurrence for this weight system
extended to permutations. Setting $C_0=1-1=0$ and using Theorem~\ref{th:Casimirs}
we can express higher Casimirs in $ZU(\GL(1|1))$ in terms of $C_1,C_2$. Namely,
we have 
$$
1-\sum_{k=0}^\infty\varphi(C_k)z^{k+1}=\frac{1-\frac{z}{1-z x_1}}{1-\frac{z}{1+z x_2}},
$$
which gives
\begin{align*}
\varphi(C_1)&=x_1+x_2,\\ \varphi(C_2)&=(x_1+x_2)(x_1-x_2+1), \\
\varphi(C_3)&=(x_1+x_2)(x_1^2-x_1x_2+x_2^2+x_1-2x_2+1)\\
            &=\varphi(C_1)(\frac{3\varphi(C_2)^2}{4\varphi(C_1)^2}+\frac{\varphi(C_1)^2}{4}-\frac{\varphi(C_1)}{2}+\frac{1}{4}), \\
            \dots
\end{align*}
so that we have
$$
C_3=\frac{3C_2^2}{4C_1}+\frac{C_1^3}{4}-\frac{C_1^2}{2}+\frac{C_1}{4},\dots
$$
and, more generally
\begin{align*}
&x=\frac{C_1^2-C_1+C_2}{2 C_1},\qquad y=\frac{C_1^2+C_1-C_2}{2 C_1},\\
&\sum _{k=0}^{\infty } C_k z^k=\frac{1-\frac{(1-(x+1) z) (y z+1)}{(1-x z) (1-(1-y) z)}}{z}=\frac{C_1 z}{\left(1-\frac{\left(-C_1^2+C_1+C_2\right) z}{2 C_1}\right) \left(1-\frac{\left(C_1^2-C_1+C_2\right) z}{2 C_1}\right)}
\end{align*}

For example, if we make this substitution in the explicit formulas for the values
of the $w_\GL$-weight system on chord diagrams whose intersection graph is a complete
graph given in~\cite{ZY2}, we obtain the following values of the $\GL(1|1)$-weight system
on these diagrams:

These results are worth to be compared with the values of the skew characteristic polynomial
of complete graphs from~\cite{DL}. 

\end{example}

\section{Proof of theorem \ref{the:main}}\label{sec:pro}
We prove the theorem by directly proving that $w_{\GL(m|n)}$ satisfies the same Recurrence Rule as $w_{\GL}$ with $C_0=m-n$.

Assuming the permutation $\sigma$ is as shown before, and suppose $\tilde\sigma$ merges the two nodes and connects the edges $a$ and $c$,
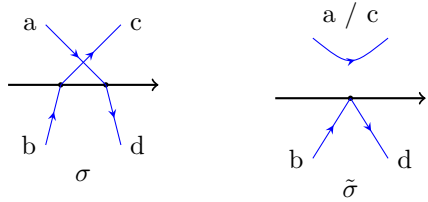
\[
  \begin{tikzpicture}[baseline={([yshift=-.5ex]current bounding box.center)},decoration={markings, mark= at position .55 with {\arrow{stealth}}}]
    \draw[->,thick] (-1,0) --  (1,0);
    \fill[black] (-.3,0) circle (1pt) ;
    \fill[black] (.3,0) circle (1pt) ;
    \draw (-.5,.8) node[left] {a};
    \draw (-.5,-.8) node[left] {b};
    \draw (.5,.8) node[right] {c};
    \draw (.5,-.8) node[right] {d};
    \draw[blue,postaction={decorate}] (-.5,.8) -- (.3,0);
    \draw[blue,postaction={decorate}] (-.3,0) -- (.5,.8);
    \draw[blue,postaction={decorate}] (-.5,-.8) -- (-.3,0);
    \draw[blue,postaction={decorate}] (.3,0) -- (.5,-.8);
    \draw (0,-1.2) node { $\sigma$};
  \end{tikzpicture}\ \ \ \ \ \ \ \ \ \ \ \ \
  \begin{tikzpicture}[baseline={([yshift=-.5ex]current bounding box.center)},decoration={markings, mark= at position .55 with {\arrow{stealth}}}]
    \draw[->,thick] (-1,0)  -- (1,0);
    \fill[black] (0,0) circle (1pt) ;
    \draw (0,.8) node[above] {a / c};
    \draw (-.5,-.8) node[left] {b};
    %\draw (.5,.8) node[right] {c};
    \draw (.5,-.8) node[right] {d};
    \draw[blue,postaction={decorate}] (-.5,.8) ..controls (0,.4) .. (.5,.8);
    \draw[blue,postaction={decorate}] (-.5,-.8) -- (0,0);
    \draw[blue,postaction={decorate}] (0,0) -- (.5,-.8);
    \draw (0,-1.2) node { $\tilde\sigma$};
  \end{tikzpicture}
\]
\begin{lemma}\label{le:2}

  We have $f_{\tilde\sigma}=f_{\sigma|\bar i_a=\bar i_c}$.
\end{lemma}
\begin{proof}\ \\
First assume the arrangement of the end points of the four arrows is as follows,
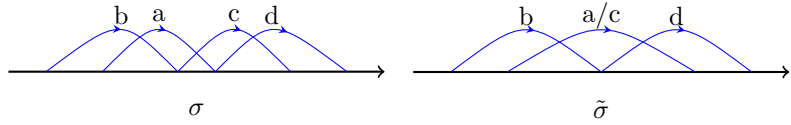
\[
  \begin{tikzpicture}[baseline={([yshift=-.5ex]current bounding box.center)},decoration={markings, mark= at position .55 with {\arrow{stealth}}},scale=2.5]
    \draw[->,thick] (-1,0) --  (1,0);
    %\fill[black] (-.3,0) circle (1pt) node[below] {\tiny k};
    %\fill[black] (.3,0) circle (1pt) node[below] {\tiny k+1};
    \draw (-.2,.3) node {a};
    \draw (-.4,.3) node {b};
    \draw (.2, .3) node {c};
    \draw (.4, .3) node {d};
    \draw[blue,postaction={decorate}] (-.8,0) ..controls(-.4,.3) .. (-.1,0);
    \draw[blue,postaction={decorate}] (-.1,0) ..controls(.2,.3) .. (.5,0);
    \draw[blue,postaction={decorate}] (-.5,0) ..controls(-.2,.3) .. (.1,0);
    \draw[blue,postaction={decorate}] (.1,0) ..controls(.4,.3) .. (.8,0);
    \draw (0,-.2) node { $\sigma$};
  \end{tikzpicture}\ \ \
    \begin{tikzpicture}[baseline={([yshift=-.5ex]current bounding box.center)},decoration={markings, mark= at position .55 with {\arrow{stealth}}},scale=2.5]
    \draw[->,thick] (-1,0) --  (1,0);
    %\fill[black] (-.3,0) circle (1pt) node[below] {\tiny k};
    %\fill[black] (.3,0) circle (1pt) node[below] {\tiny k+1};
    \draw (0,.3) node {a/c};
    \draw (-.4,.3) node {b};
    \draw (.4, .3) node {d};
    \draw[blue,postaction={decorate}] (-.8,0) ..controls(-.4,.3) .. (-0,0);
    \draw[blue,postaction={decorate}] (-.5,0)..controls(0,.3) .. (.5,0);
    \draw[blue,postaction={decorate}] (0,0) ..controls(.4,.3) .. (.8,0);
    \draw (0,-.2) node { $\tilde\sigma$};
  \end{tikzpicture}
\]Then
  \begin{enumerate}
    \item The pairs of the other edges except $a,b,c,d$ are not changed.
    \item The pairs of $b,d$ with the other edges are not changed.
    \item The other edges making a distinguished pair with only $a$ or $c$
    will also make a distinguished pair with $a/c$. And the edges making distinguished pairs
    both with $a$ and $c$ will not make a distinguished pair with $a/c$. However,
     since $\bar i_a=\bar i_c$ and the field is $\BZ_2$, we have $\bar i_x\bar i_a+\bar i_x\bar i_c=0$. Therefore, these cases do not differ.
    \item We only need to consider the relationship between $a,b,c,d$
      \begin{enumerate}
        \item the linear term: since the edges $a$ and $c$ turn into a longer edge, the difference in the linear term is $\bar i_a$.
        \item the difference in the quadratic term is $\bar i_a\bar i_c=\bar i_a$ as well.
      \end{enumerate}
  \end{enumerate}
Summing everything together, we obtain no difference.

For the other arrangements of the end points of the four arrows the calculation similar, and we skip the rest of the proof.
\end{proof}

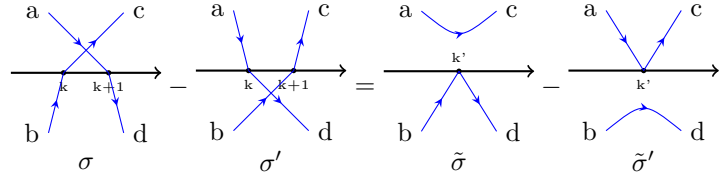
\begin{equation*}
  \begin{tikzpicture}[baseline={([yshift=-.5ex]current bounding box.center)},decoration={markings, mark= at position .55 with {\arrow{stealth}}}]
    \draw[->,thick] (-1,0) --  (1,0);
    \fill[black] (-.3,0) circle (1pt) node[below] {\tiny k};
    \fill[black] (.3,0) circle (1pt) node[below] {\tiny k+1};
    \draw (-.5,.8) node[left] {a};
    \draw (-.5,-.8) node[left] {b};
    \draw (.5,.8) node[right] {c};
    \draw (.5,-.8) node[right] {d};
    \draw[blue,postaction={decorate}] (-.5,.8) -- (.3,0);
    \draw[blue,postaction={decorate}] (-.3,0) -- (.5,.8);
    \draw[blue,postaction={decorate}] (-.5,-.8) -- (-.3,0);
    \draw[blue,postaction={decorate}] (.3,0) -- (.5,-.8);
    \draw (0,-1.2) node { $\sigma$};
  \end{tikzpicture}-
  \begin{tikzpicture}[baseline={([yshift=-.5ex]current bounding box.center)},decoration={markings, mark= at position .55 with {\arrow{stealth}}}]
    \draw[->,thick] (-1,0) --  (1,0);
    \fill[black] (.3,0) circle (1pt) node[below] {\tiny k+1};
    \fill[black] (-.3,0) circle (1pt) node[below] {\tiny k};
    \draw (-.5,.8) node[left] {a};
    \draw (-.5,-.8) node[left] {b};
    \draw (.5,.8) node[right] {c};
    \draw (.5,-.8) node[right] {d};
    \draw[blue,postaction={decorate}] (-.5,.8) -- (-.3,0);
    \draw[blue,postaction={decorate}] (.3,0) -- (.5,.8);
    \draw[blue,postaction={decorate}] (-.5,-.8) -- (.3,0);
    \draw[blue,postaction={decorate}] (-.3,0) -- (.5,-.8);
    \draw (0,-1.2) node { $\sigma'$};
  \end{tikzpicture}=
  \begin{tikzpicture}[baseline={([yshift=-.5ex]current bounding box.center)},decoration={markings, mark= at position .55 with {\arrow{stealth}}}]
    \draw[->,thick] (-1,0)  -- (1,0);
    \fill[black] (0,0) circle (1pt) node[above] {\tiny k'};
    \draw (-.5,.8) node[left] {a};
    \draw (-.5,-.8) node[left] {b};
    \draw (.5,.8) node[right] {c};
    \draw (.5,-.8) node[right] {d};
    \draw[blue,postaction={decorate}] (-.5,.8) ..controls (0,.4) .. (.5,.8);
    \draw[blue,postaction={decorate}] (-.5,-.8) -- (0,0);
    \draw[blue,postaction={decorate}] (0,0) -- (.5,-.8);
    \draw (0,-1.2) node { $\tilde\sigma$};
  \end{tikzpicture}-
  \begin{tikzpicture}[baseline={([yshift=-.5ex]current bounding box.center)},decoration={markings, mark= at position .55 with {\arrow{stealth}}}]
    \draw[->,thick] (-1,0) --  (1,0);
    \fill[black] (0,0) circle (1pt) node[below] {\tiny k'};
    \draw (-.5,.8) node[left] {a};
    \draw (-.5,-.8) node[left] {b};
    \draw (.5,.8) node[right] {c};
    \draw (.5,-.8) node[right] {d};
    \draw[blue,postaction={decorate}] (-.5,-.8) ..controls (0,-.4) .. (.5,-.8);
    \draw[blue,postaction={decorate}] (-.5,.8) -- (0,0);
    \draw[blue,postaction={decorate}] (0,0) -- (.5,.8);
    \draw (0,-1.2) node { $\tilde\sigma'$};
  \end{tikzpicture}
\end{equation*}

Now let us look back at the Lie superbracket
\[
   [E_{ij},E_{kl}]=E_{ij}E_{kl}-(-1)^{(\bar i +\bar j)(\bar k +\bar l)}E_{kl}E_{ij}=\delta_{jk}E_{il}-(-1)^{(\bar i +\bar j)(\bar k +\bar l)}\delta_{il}E_{kj}
\]
We produce everything we need to let the first term be $w_{\GL(m|n)}(\sigma)$:
\begin{align*}
  & \sum(-1)^{f_\sigma}\cdots E_{ij}E_{kl}\cdots-\sum (-1)^{f_\sigma+(\bar i +\bar j)(\bar k +\bar l)}\cdots E_{kl}E_{ij}\cdots\\
   =&\sum \delta_{jk}(-1)^{f_\sigma}\cdots E_{il}\cdots-\sum \delta_{il}(-1)^{f_\sigma+(\bar i +\bar j)(\bar k +\bar l)}\cdots E_{kj}\cdots
\end{align*}
Using Lemma \ref{le:1} and Lemma \ref{le:2}, we obtain
\begin{align*}
  & \sum(-1)^{f_\sigma}\cdots E_{ij}E_{kl}\cdots-\sum (-1)^{f_{\sigma'}}\cdots E_{kl}E_{ij}\cdots\\
   =&\sum (-1)^{f_{\tilde\sigma}}\cdots E_{il}\cdots-\sum(-1)^{f_{\tilde\sigma'}}\cdots E_{kj}\cdots'
\end{align*}
which is
\[
  w_{\GL(m|n)}(\sigma)-w_{\GL(m|n)}(\sigma')=w_{\GL(m|n)}(\tilde\sigma)-w_{\GL(m|n)}(\tilde\sigma').
\]
It is the same Recurrence Rule as for $w_{\GL(m-n)}$.

For the special case $\sigma(k+1)=k$, the recurrence looks like
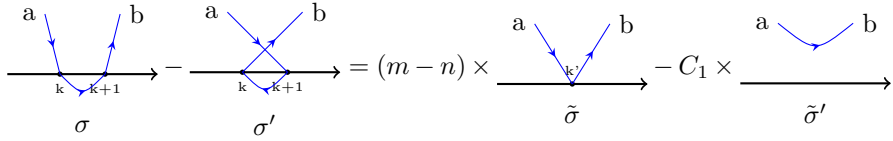
\begin{equation*}
  \begin{tikzpicture}[baseline={([yshift=-.5ex]current bounding box.center)},decoration={markings, mark= at position .55 with {\arrow{stealth}}}]
    \draw[->,thick] (-1,0) --  (1,0);
    \fill[black] (.3,0) circle (1pt) node[below] {\tiny k+1};
    \fill[black] (-.3,0) circle (1pt) node[below] {\tiny k};
    \draw (-.5,.8) node[left] {a};
    \draw (.5,.8) node[right] {b};
    \draw[blue,postaction={decorate}] (-.5,.8) -- (-.3,0);
    \draw[blue,postaction={decorate}] (.3,0) -- (.5,.8);
    \draw[blue,postaction={decorate}] (-.3,0) ..controls(0,-.3).. (.3,0);
    \draw (0,-0.7) node { $\sigma$};
  \end{tikzpicture}-
  \begin{tikzpicture}[baseline={([yshift=-.5ex]current bounding box.center)},decoration={markings, mark= at position .55 with {\arrow{stealth}}}]
    \draw[->,thick] (-1,0) --  (1,0);
    \fill[black] (-.3,0) circle (1pt) node[below] {\tiny k};
    \fill[black] (.3,0) circle (1pt) node[below] {\tiny k+1};
    \draw (-.5,.8) node[left] {a};
    \draw (.5,.8) node[right] {b};
    \draw[blue,postaction={decorate}] (-.5,.8) -- (.3,0);
    \draw[blue,postaction={decorate}] (-.3,0) -- (.5,.8);
    \draw[blue,postaction={decorate}] (.3,0) ..controls(0,-.3).. (-.3,0);
    \draw (0,-0.7) node { $\sigma'$};
  \end{tikzpicture}=(m-n)\times
  \begin{tikzpicture}[baseline={([yshift=-.5ex]current bounding box.center)},decoration={markings, mark= at position .55 with {\arrow{stealth}}}]
    \draw[->,thick] (-1,0) --  (1,0);
    \fill[black] (0,0) circle (1pt) node[above] {\tiny k'};
    \draw (-.5,.8) node[left] {a};
    \draw (.5,.8) node[right] {b};
    \draw[blue,postaction={decorate}] (-.5,.8) -- (0,0);
    \draw[blue,postaction={decorate}] (0,0) -- (.5,.8);
    \draw (0,-0.4) node { $\tilde\sigma$};
    \end{tikzpicture}-C_1\times
  \begin{tikzpicture}[baseline={([yshift=-.5ex]current bounding box.center)},decoration={markings, mark= at position .55 with {\arrow{stealth}}}]
    \draw[->,thick] (-1,0)  -- (1,0);
    \draw (-.5,.8) node[left] {a};
    \draw (.5,.8) node[right] {b};
    \draw[blue,postaction={decorate}] (-.5,.8) ..controls (0,.4) .. (.5,.8);
    \draw (0,-0.4) node { $\tilde\sigma'$};

  \end{tikzpicture}
\end{equation*}
We have the relationship
\[
  E_{ij}E_{jk}-(-1)^{(\bar i +\bar j)(\bar j +\bar k)}E_{jk}E_{ij}=\delta_{jj}E_{ik}-(-1)^{(\bar i +\bar j)(\bar j +\bar k)}\delta_{ik}E_{jj}
\]
We produce everything we need to let the first term be $w_{\GL(m|n)}(\sigma)$:
\begin{align*}
  & \sum(-1)^{f_\sigma}\cdots E_{ij}E_{jk}\cdots-\sum (-1)^{f_\sigma+(\bar i +\bar j)(\bar j +\bar k)}\cdots E_{jk}E_{ij}\cdots\\
   =&\sum \delta_{jj}(-1)^{f_\sigma}\cdots E_{ik}\cdots-\sum \delta_{ik}(-1)^{f_\sigma+(\bar i +\bar j)(\bar j +\bar k)}\cdots E_{jj}\cdots
\end{align*}
We have
\begin{align*}
  & \sum(-1)^{f_\sigma}\cdots E_{ij}E_{jk}\cdots-\sum (-1)^{f_{\sigma'}}\cdots E_{jk}E_{ij}\cdots\\
   =&\sum (-1)^{\bar j +f_{\tilde\sigma}}\cdots E_{ik}\cdots-\sum(-1)^{f_{\tilde\sigma'}}\cdots E_{jj}\cdots\\
   =&(m-n)\sum (-1)^{f_{\tilde\sigma}}\cdots E_{ik}\cdots-C_1\sum(-1)^{f_{\tilde\sigma'}}\cdots 1\cdots,
\end{align*}
hence $w_{\GL(m|n)}$ satisfies the special case of the Recurrence Rule with number $m-n$.

Finally, $w_{\GL(m|n)}$ obeys the same Recurrence Rule as $w_{\GL}$ with $C_0=m-n$.

\end{document}